\documentclass[12pt]{amsart}
\usepackage{amsfonts}
\usepackage{fullpage}
\usepackage{amsmath,amsthm}
\usepackage{latexsym}
\usepackage{array}
\usepackage{dsfont}
\usepackage{amssymb}
\usepackage{graphics}
\usepackage{enumerate}
\usepackage[english]{babel}

\usepackage{color}
\definecolor{marin}{rgb}   {0.,   0.3,   0.7} 
\definecolor{rouge}{rgb}   {0.8,   0.,   0.} 
\definecolor{sepia}{rgb}   {0.8,   0.5,   0.} 
\usepackage[colorlinks,citecolor=marin,linkcolor=rouge,
            bookmarksopen,
            bookmarksnumbered
           ]{hyperref}

\newtheorem{lemma}{Lemma}[section]

\newtheorem{theorem}[lemma]{Theorem}
\newtheorem{proposition}[lemma]{Proposition}
\newtheorem{corollary}[lemma]{Corollary}
\newtheorem{remark}[lemma]{Remark}
\newtheorem{example}[lemma]{Example}

\newtheorem{notation}[lemma]{Notation}
\newtheorem{definition}[lemma]{Definition}
\newtheorem{conclusion}[lemma]{Conclusion}
\numberwithin{equation}{section}
\newcommand{\QED}{\mbox{}\hfill \raisebox{-0.2pt}{\rule{5.6pt}{6pt}\rule{0pt}{0pt}} 
          \medskip\par}

\newcommand{\eps}{\varepsilon}

\newcommand{\dd}{\mathrm{d}}

\newcommand{\Hc}{\mathcal{H}}

\newcommand{\R}{\mathbb{R}}

\newcommand{\T}{\mathbb{T}}
\newcommand{\Lc}{\mathcal{L}}

\newcommand{\Z}{\mathbb{Z}}

\newcommand{\Norm}[2]{\|#1\|\left.\vphantom{T_{j_0}^0}\!\!\right._{#2}}

\author{Erwan Faou}
\address{INRIA \& ENS Cachan Bretagne  \\
Avenue Robert Schumann F-35170 Bruz, France. } 
\email{Erwan.Faou@inria.fr}

\author{ Fr\'ed\'eric Rousset }
\address{ Laboratoire de Math\'ematiques d'Orsay (UMR 8628) Universit\'e Paris-Sud et Institut Universitaire de France}
  \email{frederic.rousset@math.u-psud.fr}

\title[Landau damping   for the Vlasov-HMF model]
{ Landau damping in Sobolev spaces \\ for the Vlasov-HMF model}

\begin{document}

\begin{abstract}
We consider the Vlasov-HMF (Hamiltonian Mean-Field) model. We consider solutions starting in a small Sobolev neighborhood of a  spatially homogeneous state  satisfying
 a linearized stability criterion (Penrose criterion). 
We prove that  these  solutions exhibit  a scattering behavior to a modified state, which implies a nonlinear Landau damping effect with polynomial rate of damping. \end{abstract}

\subjclass{ 35Q83, 35P25 }
\keywords{Vlasov equations, Scattering, Landau damping, HMF model}
\thanks{
}

\maketitle

\section{Introduction}

In this paper we consider the Vlasov-HMF model. This model has received much interest in the physics
literature for many reasons:  It  is a simple ideal toy model that keeps several features of the long range interactions,  it is 
a simplification of physical systems like charged or gravitational sheet models and it is rather easy to make numerical simulations on it. We refer for example  to \cite{Bouchet1}, \cite{Bouchet2}, \cite{Bouchet3}, \cite{Caglioti-Rousset1}, \cite{Caglioti-Rousset2}   for more details.

We shall study the long time behavior of solutions to this model for initial data that are  small perturbations in a  weighted  Sobolev space to a  spatially homogeneous stationary state satisfying a Penrose type
stability condition. We shall prove that   the solution scatters when times goes to infinity towards a modified state close to the initial data in a Sobolev space of lower order. This result implies a nonlinear Landau damping effect with polynomial rate  for the solution which converges weakly towards a modified spatially homogeneous state. 
In  the case of analytic  or Gevrey regularity, this result has been shown  to hold for a large class of Vlasov equations that contains the Vlasov-Poisson system   by Mouhot and Villani \cite{Mouhot-Villani} (see also the recent simplified proof
 \cite{Bedrossian-Masmoudi-Mouhot}). Some earlier partial results were obtained in \cite{Caglioti-Maffei}, \cite{hwang}. The related problem of the stability of the Couette flow in the two-dimensional Euler equation has been also studied recently
  \cite{Bedrossian-Masmoudi}.  
    The  question  left open  in these papers is  the possibility of nonlinear Landau damping for Sobolev perturbations.
      In this case, one cannot hope for an exponential damping, but we can wonder if it could occur by allowing  polynomial rates.
     For the Vlasov-Poisson system, it was proven in \cite{Lin} that this is false
      with rough Sobolev regularity due to the presence of  arbitrarily close travelling BGK states.  Nevertheless, this obstruction disappears for sufficiently high Sobolev regularity
      as also proven in \cite{Lin}. These  arguments can also  be extended to a large class of Vlasov equations and in particular the HMF model. Note that   the regularity of the interaction kernel does not play an essential part in the argument of \cite{Mouhot-Villani} (though the decay in Fourier space provided by the  Coulomb interaction seems critical), 
       it is more the nonlinear ``plasma-echo" effect which is crucial to handle.
  Besides the physical interest of the HMF model, it is thus mathematically interesting  to study the  possibility of nonlinear Landau-damping in Sobolev spaces for this simple model.

\subsection{The Vlasov-HMF model}
 The Vlasov-HMF model reads
\begin{equation}
\label{eq:hmf1}
\partial_t f (t,x,v)+  v \partial_x f (t,x,v) =  \partial_{x}
 \Big(\int_{\R \times \mathbb{T}} P(x-y) f(t,y,u) \dd u \dd y \Big) \partial_v f(t,x,v), 
\end{equation}
where $(x,v) \in \T \times \R$ and  the kernel $P(x)$ is given by  $P(x) = \cos(x)$. Note that the main difference with the Vlasov-Poisson equation is the regularity of the kernel: in this latter case, $P(x) = \sum_{k \geq 0} k^{-2} \cos(k x)$ is the kernel associated with the inverse of the Laplace operator.  The HMF model is thus the simplest nonlinear model with the structure \eqref{eq:hmf1}. We consider  initial data under the form $ f_{0}(x,v)= \eta(v) + \eps r_{0}(x,v)$
 where $\eps$ is a small parameter and $r_{0}$ is of size one (in a suitable functional space). This means that we study small perturbations of a stationary solution $\eta(v)$.
 We shall thus  write  the solution at time $t$ under the form
$$
f(t,x,v) = \eta(v) + \varepsilon r(t,x,v). 
$$
 We are interested in the study of the behavior of $f$ when time goes to infinity. 
 To filter the effect of the free transport, it is convenient to introduce (as in \cite{Mouhot-Villani}, \cite{Bedrossian-Masmoudi-Mouhot}) the  unknown
 $g(t,x,v ) = r(t,x + tv,v)$  that is solution of the equation
\begin{equation}
\label{eq:vp3}
\partial_t g = \{Ê\phi(t,g), \eta\}  + \varepsilon \{ \phi(t,g),g \}. 
\end{equation}
where 
\begin{equation}
\label{eq:phi}
\phi(t,g) (x,v)= \int_{\R \times \mathbb{T}} ( \cos(x-y + t(v-u)) ) g(t,y,u) \dd u \dd y
\end{equation}
 and  $\{f,g\} = \partial_x f \partial_v g - \partial_v f \partial_x g$ is the usual microcanonical Poisson bracket. 
 We shall  usually write $\phi(t)$ when the dependence in $g$ is clear.

We shall work in the following weighted Sobolev spaces, 
 for 
$m_0 > 1/2$ be given, we set
\begin{equation}
\label{defsob}
\Norm{f}{\Hc^n}^2 = \sum_{|p| + |q| \leqslant n} \int_{\T \times \R} (1 + |v|^2)^{m_0} | \partial_x^p \partial_v^q f|^2 \dd x \dd v , 
\end{equation}
and we shall denote by  $\Hc^n$ the corresponding function space.  Note that compared to the usual Sobolev space $H^n$, there is also a fixed weight
 $(1+|v|^2)^{m_{0}}$ in physical space. The interest of the weight is that functions in $\Hc^0$ are in $L^1$ and thus  it allows to get a pointwise control in Fourier  (see  Lemma \ref{lememb} below).  We do not include the dependence in $m_{0}$ in the notation since $m_{0}$ will be fixed. 
%
%
We  shall  denote  by $\hat \cdot$ or $\mathcal{F}$ the Fourier transform  on $\mathbb{T} \times \mathbb{R}$  given by  
$$
\hat f_k(\xi) = \frac{1}{2 \pi} \int_{\T \times \R} f(x,v) e^{-ikx - i \xi v} \dd x \dd v.
$$
%
%

%
Note that due to the regularity of the interaction kernel and the conservation of the $L^p$ norms, it is very easy to prove the global well-posedness of the Vlasov-HMF model  in 
 $\Hc^{s}$ for every $s \geq 0$. Nevertheless, in order to study the asymptotic behaviour of $g$, the regularity of the kernel is not of obvious help. Indeed, 
 when performing energy estimate on \eqref{eq:vp3}, it costs one positive power of $t$  each time  one puts a $v$ derivative on the kernel.

\subsection{ The Penrose criterion}
\label{sectionpenrose}
 We shall need a stability property of the reference state $\eta$ in order to control the linear part of the Vlasov equation \eqref{eq:vp3}.
 Let us denote by $\eta$, the  spatially homogeneous stationary state and let us define the function
$$ 
K(n,t)= - np_{n}\,  nt \,   \hat{\eta}_{0} (nt) \mathds{1}_{t \geq 0}, \quad t \in \mathbb{R}, \quad n \in \Z,
$$
where $(p_{k})_{k \in \Z}$ are the Fourier coefficients of the kernel $P(x)$. We shall denote by $\hat{K}(n,\tau)= \int_{\mathbb{R}} e^{-i \tau t} K(n,t) \, dt$ the Fourier
transform of $K(n, \cdot)$. 
We shall assume that  $\eta$ satisfies the following condition 
\begin{equation}
\nonumber
{\bf (H)}  \quad  (1 + v^2 ) \eta(v) \in \Hc^5 \quad  \mbox{and } \quad  \exists\,  \kappa>0, \quad   \inf_{ \mathrm{Im}\, \tau \leq 0 }  | 1-  \hat{K}(n,\tau) | \geq \kappa, \quad n= \pm1.
\end{equation}
Note that  thanks to the localization property of $\eta$ in the first part of the assumption, the Fourier transform of $K$ can be indeed continued in the half plane
  $\mathrm{Im}\, \tau \leq 0$.
 Here, the assumption is particularly simple due to the fact that for our kernel, there are only two non-zero Fourier
 modes.
This assumption is very similar to the one used in \cite{Mouhot-Villani}, \cite{Bedrossian-Masmoudi-Mouhot} and can be related to the standard  statement of the Penrose criterion.
 In particular it is verified  for the states $\eta(v)= \rho(|v|)$ with $\rho$ non-increasing which  are also known to be Lyapounov stable for the nonlinear equation
 (see \cite{Marchioro-Pulvirenti}).

\subsection{Main result}

In the evolution of the solution $g(t,x,v)$ of \eqref{eq:vp3}, 
an important role is played by the quantity
\begin{equation}
\label{eq:zetak}
\zeta_k(t) = \hat{g}_k(t,kt), \quad k  \in \{ \pm 1 \}, 
\end{equation}
such that 
$$
\phi(t,g) = \frac12  \sum_{k \in \{ \pm 1\}}e^{ikx}e^{iktv} \zeta_k(t).
$$
Note that  for $k \neq 0$,   $ \zeta_{k}(t)$ is the Fourier coefficient in $x$ of the density $\rho(t,x)= \int_{\mathbb{R}} f(t,x,v) \, dv$.
  This  quantity also plays a key part in the analysis of 
 \cite{Mouhot-Villani}, \cite{Bedrossian-Masmoudi-Mouhot}. Note that  here we need only  to control two Fourier modes due to our simple interaction kernel.

Let us  define for every $s \geq 4$ and $T \geq 0$ the  weighted norm
  $$  Q_{T,s}(g)= \sup_{t \in [0,  T]}  {\| g(t) \|_{\Hc^s} \over \langle t \rangle^3} +   \sup_{t \in [0,T]} \sup_{k \in \{\pm 1\}} \langle t \rangle^{s-1}|\zeta_{k}(t) |  +  \sup_{t \in [0, T ]} \|g(t) \|_{\Hc^{s-4}}.
$$
Our main result is:
\begin{theorem}
\label{maintheo}
Let us fix $s \geq 7$ and $R_{0}>0$ such that $Q_{0, s}(g) \leq R_{0}$ and assume that $\eta\in \Hc^{s+4}$ satisfies the assumption ${\bf (H)}$.  Then there exists $R>0$ and $\eps_{0}>0$ such that for every $\eps \in (0, \eps_{0}]$ and
 for every $T\geq 0$, we have the estimate
 $$ Q_{T, s}(g) \leq R.$$
\end{theorem}
As a consequence, we obtain the following scattering result:
\begin{corollary}
\label{Landau}
Under the assumption of Theorem \ref{maintheo}, there exists a constant $C$ and $g^\infty(x,v) \in \Hc^{s-4}$ such that for all $r \leq s - 4$ and $r \geq 1$, 
\begin{equation}
\label{scat}
\forall\, t \geq 0, \quad \Norm{g(t,x,v) - g^{\infty}(x,v)}{\Hc^{r}}  \leq \frac{C}{\langle t \rangle^{s - r-3}}. 
\end{equation}
\end{corollary}

The consequence of such results is the following nonlinear Landau damping effect: as $g(t,x,v)$ is bounded in $\Hc^{s-4}$, the solution $f(t,x,v) = \eta(v) + \varepsilon r(t,x,v) = \eta(v) + \varepsilon g(t,x-tv,v)$ satisfies 
$$
\forall\, n \in \Z^*, \quad \forall \xi \in \R,\quad \forall\, \alpha + \beta  = s-4, \quad  |\hat f_n(t,\xi) | = \varepsilon|Ê\hat g_n(t,\xi + nt )  |\leq \frac{C \varepsilon}{\langle \xi + nt \rangle^{\alpha} \langle n \rangle^\beta}. 
$$
The last estimate being a consequence of  the elementary embedding  Lemma \ref{lememb}.
 This yields that  for every $n \neq 0$, $ \hat f_n(t,\xi)$ tends to zero with a polynomial rate.
 
Moreover,  by setting 
$$
\eta^\infty(v) := \eta(v) + \frac{\varepsilon}{2\pi}\int_\T g^{\infty}(x,v) \dd x, 
$$
we have by the previous corollary (and again  Lemma \ref{lememb}) that  for $r \leq s - 4$, 
$$
\forall\, \xi \in \R, \quad 
|\hat f_0(\xi) - \hat \eta^{\infty}_0(\xi) | \leq \frac{C}{\langle \xi \rangle^{r} \langle t \rangle^{s- r - 3}}
$$
In other words,  $f(t,x,v)$ converges weakly towards $\eta^\infty(v)$. 
\bigskip

The remaining of the paper is devoted to the proof of Theorem \ref{maintheo}. We shall obtain Corollary \ref{Landau} in section \eqref{ouf} as an easy consequence.
 As pointed out in \cite{Mouhot-Villani} the control of the ``plasma echoes" that can be seen as  kind of resonances  is crucial to prove nonlinear Landau damping. These resonances
  occur when $nt= k\sigma$ in the last  integral term of \eqref{zetanp1}. The main structural property of the Vlasov-HMF model that makes possible the following short  proof of nonlinear
  Landau damping in Sobolev spaces is that the resonances are easy to analyze, the only possibility is when $n=k =  \pm 1$ and  $t = \pm \sigma$.  Moreover, the structure of the nonlinearity
   then allows  to control it  without loss of decay.  Making an analogy with  dispersive equations (see \cite{Klainerman} for example), the nonlinearity of  the Vlasov-HMF model could be thought  as a  nonlinearity with null structure. 
    Our approach also allows to handle the case of a kernel with a finite number of modes  which  allows more resonances, we briefly sketch the modification in section \ref{sectionfinite}. Nevertheless, it is still unclear
     if this can be done for the Vlasov-Poisson equation.


\section{A priori estimates}

%
%
%

In this section, we shall study a priori estimates for the solution of \eqref{eq:vp3}.
 Let us fix $s \geq 7$ and introduce the  weighted norms:
 \begin{equation}
 \label{eq:NetM}  N_{T,s}(g)= \sup_{t \in [0,  T]}  {\| g(t) \|_{\Hc^s} \over \langle t \rangle^3}, \quad   M_{T, \gamma}(\zeta)= \sup_{t \in [0,T]} \sup_{k \in \{\pm 1\}} \langle t \rangle^\gamma |\zeta_{k}(t) |
 \end{equation}
 so that 
  \begin{equation}
  \label{Qdef}  Q_{T,s}(g)=  N_{T,s}(g) + M_{T, s-1}(\zeta) + \sup_{[0, T ]} \|g(t) \|_{\Hc^{s-4}}.
  \end{equation}
   Let us take  $R_{0}>0$ such that $Q_{0,s} (g) \leq R_{0}$. Our aim is to prove that when $\eps$ is sufficiently small we  can choose $R$ so that we have
   $$ Q_{T, s}(g) \leq  R$$  for every $T \geq 0$.
   
   In the following a priori estimates, $C$ stands for a number which may change from line to line and which is independent of $R_{0}$, $R$, $\eps$ and $T$.
   
   We shall  make constant use of the following elementary  lemma.
\begin{lemma}
\label{lememb}
For every $\alpha, \, \beta, \, n\, \in \mathbb{N}$ with  $\alpha + \beta = n$ we have the following inequality: 
\begin{equation}
\label{emb1}
\forall\, k \in \Z, \quad \forall\, \xi \in \R, \quad 
|\hat{f}_k(\xi)| \leqslant 2^{n/2}C(m_0) \langle k\rangle^{- \alpha} \langle \xi\rangle^{- \beta} \Norm{f}{\Hc^{n}},
\end{equation}
where $C(m_0)$ depends only on $m_0$ and where $\langle x \rangle = (1 + |x|^2)^{1/2}$ for $x \in \R$. 
\end{lemma}
\begin{proof}
We have by  using the  Cauchy-Schwarz inequality that 
\begin{eqnarray*}
\big| k^\alpha \xi^\beta  \hat{f}_k(\xi) \big| &=&  \frac{1}{2\pi} \left|\int_{\T \times \R} \partial_x^\alpha \partial_v^\beta f(x,v) e^{-ikx} e^{-iv\xi} \dd x  \dd v\right| \\
&\leqslant& C  \Norm{f}{\Hc^n} \Big( \int_\R (1 + |v|^2)^{-m_0} \dd v \Big)^{1/2}. 
\end{eqnarray*}
The previous inequality with $\alpha = \beta = 0$ yields the result when $k = 0$ or $|\xi | \leqslant 1$ and we conclude by using 
$\langle x \rangle \leqslant 2^{\alpha/2} |x|^2$ for $|x| > 1$. 
\end{proof}
%

\subsection{Estimate of $M_{T, s-1}(\zeta)$ }
Towards the proof of Theorem \ref{maintheo}, we shall first estimate $\zeta_{k}(t)$, $k=\pm1$.
 \begin{proposition}
 \label{propzeta}
 Assuming that $\eta \in \Hc^{s+2}$ verifies the assumption {\bf (H)}, then there exists $C>0$ such that for every $T>0$, every  solution of \eqref{eq:vp3}  such that
 $ Q_{T, s} (g) \leq R$ enjoys the estimate
 \begin{equation}
 \label{youyou} 
 M_{T, s-1}(\zeta) \leq C \big(  R_{0} + \eps R^2 \big).
 \end{equation}
 \end{proposition}

\begin{proof}
 We first note that in Fourier space, the equation \eqref{eq:vp3} can be written after integration in time, 
\begin{multline}
\label{fourier1}
\hat g_n(t,\xi) = \hat g_n(0,\xi)  +  \int_{0}^t p_n \zeta_n(\sigma) \hat \eta_0(\xi - n \sigma)( n^2 \sigma - n \xi)Ê\dd \sigma
 \\
+ \varepsilon \sum_{k \in \Z } p_k \int_0^t \zeta_k(\sigma)  \hat g_{n-k}(\sigma,\xi - k\sigma) (  n k \sigma - k \xi)Ê\dd \sigma, 
\end{multline}
for all $(n,\xi) \in \Z \times \R$, 
with $p_k = \frac12$ for $k \in \{\pm 1\}$ and $p_k = 0$ for $k \neq \pm 1$, and where the $\zeta_k(t)$ are defined by \eqref{eq:zetak}. 
Setting $\xi = n t$ in \eqref{fourier1}, the equation satisfied by $(\zeta_n(t))_{n=\pm1}$ can be written under the almost closed form 
\begin{multline}
\label{zetanp1}
\zeta_n(t) = \hat g_n(0,nt) -  \int_{0}^t p_n \zeta_n(\sigma)  \hat \eta_0(n(t - \sigma)) n^2 ( t- \sigma )Ê\dd \sigma\\
- \varepsilon   \sum_{k \in \{ \pm 1\}} p_k \int_0^t \zeta_k(\sigma)  \hat g_{n-k}(\sigma, nt - k\sigma) k n(t-\sigma) Ê\dd \sigma. 
\end{multline}

To study the equation \eqref{zetanp1}, we shall first consider the corresponding  linear equation, that is to say that we  shall first see
 \begin{equation}
 \label{Fdef} F_{n}(t):=   \hat g_n(0,nt)- \varepsilon   \sum_{k \in \{ \pm 1\}} p_k \int_0^t \zeta_k(\sigma)  \hat g_{n-k}(\sigma, nt - k\sigma) k n(t-\sigma) Ê\dd \sigma
 \end{equation}
  as a given source term and we shall study the linear integral equation
 \begin{equation}
 \label{volterra1}
  \zeta_{n} (t) =  \int_{0}^t K(n, t-\sigma) \zeta_{n}(\sigma) \,\dd\sigma + F_{n}(t) \quad n= \pm 1
 \end{equation}
 where the kernel $K(n,t)$ has been introduced in section \ref{sectionpenrose}.
 
  For this linear equation, we have the estimate:
  \begin{lemma}
  \label{lemvolterra}
 Let $\gamma \geq 0$, and  assume that $\eta \in \Hc^{\gamma+3}$ satisfies ${\bf (H)}$. Then,
    there exists $C>0$ such   for every $T \geq 0$, we have
   $$ M_{T, \gamma}( \zeta) \leq C M_{T, \gamma}( F).$$
     \end{lemma}

Let us postpone the proof of the Lemma and finish the proof of Proposition \ref{propzeta}.

 From the previous Lemma and \eqref{emb1}, we first  get that
 \begin{equation}
 \label{est1:1} M_{T, s-1}(\zeta) \leq  C\big( \|g(0) \|_{s} + \eps  M_{T, s-1}(F^1) + \eps  M_{T, s-1}(F^2) \big) 
 \end{equation}
  with 
 where $F^1$ corresponds to the term with $k= - n$ in \eqref{Fdef} and $F^2$ corresponds to the term with $k=n$, hence
 \begin{align*}
 &  F^1_{n}(t) = - n^2 p_{-n} \int_{0}^t \zeta_{-n}(\sigma) \hat g_{2n}( \sigma, n(t+ \sigma)) (t- \sigma) \, \dd\sigma, \quad n= \pm1, \\
 &  F^2_{n}(t) = n^2 p_{n} \int_{0}^t \zeta_{n}(\sigma) \hat g_{0}( \sigma, n(t- \sigma)) (t- \sigma) \,\dd\sigma, \quad n= \pm 1.
 \end{align*}
 Let us estimate $F^1_{n}$, by using  again \eqref{emb1} and the definition \eqref{eq:NetM} of $N_{\sigma,s}$,  we get that
 $$ |F^1_{n}(t)| \leq  C \int_{0}^t {  (t- \sigma) \langle \sigma \rangle^3 M_{\sigma, s-1}(\zeta) N_{\sigma, s}(g)  \over \langle \sigma \rangle^{s-1}   \langle t+  \sigma \rangle^s} \, \dd \sigma \leq C {R^2 \over\langle  t \rangle^{s-1}} \int_{0}^{+ \infty} { 1 \over \langle \sigma \rangle^{s-4}}\, \dd \sigma \leq   C {R^2 \over\langle  t \rangle^{s-1}}$$
  provided  $s \geq  6$. This yields that for all $T \geq 0$
  $$ M_{T, s-1}(F^1) \leq CR^2.$$
  
  To estimate $F_{n}^2$, we split the integral into two parts: we write
   $$  F^2_{n}(t) =  I^1_{n}(t) + I_{n}^2 (t)$$ 
   with
    \begin{align*}
 &    I^1_{n}(t) = n^2 p_{n} \int_{0}^{t \over 2} \zeta_{n}(\sigma) \hat g_{0}( \sigma, n(t- \sigma)) (t- \sigma) \, \dd\sigma, \quad n= \pm 1, \\
  &    I^2_{n}(t) = n^2 p_{n} \int_{t \over 2}^{t } \zeta_{n}(\sigma) \hat g_{0}( \sigma, n(t- \sigma)) (t- \sigma) \, \dd\sigma, \quad n= \pm 1.
     \end{align*}
      For $I_{n}^1$, we proceed as previously, 
      $$ |I_{n}^1 (t) | \leq C R^2  \int_{0}^{t \over 2 } {  \langle \sigma \rangle^3( t- \sigma)  \over \langle \sigma \rangle^{s-1}  \langle t-\sigma \rangle^{s}} \, \dd \sigma \leq  { C R^2 \over \langle t \rangle^{s-1}} \int_{0}^{+ \infty} { 1 \over \langle
       \sigma \rangle^{s- 4}} \, \dd \sigma$$
        and hence since $s \geq 6$, we have
   $$ M_{T, s-1}(I^1) \leq C R^2.$$
     To estimate $I_{n}^2$, we shall rather use the last factor in the definition of $Q_{s, T}$ in \eqref{Qdef}. By using again \eqref{emb1},  we write
     $$ | I_{n}^2(t) | \leq \int_{t \over 2}^{ t  }  { M_{\sigma, s-1}(\zeta) \over \langle \sigma \rangle^{s-1} }   {\| g(\sigma) \|_{\Hc^{s-4}} \over  \langle  t- \sigma \rangle^{s-5}} \, \dd \sigma
      \leq  {C R^2 \over \langle t \rangle^{s-1}  }\int_{0}^{+ \infty} { 1 \over \langle \tau \rangle^{s-5}}\, \dd \sigma \leq { CR^2 \over \langle t \rangle^{s-1}}$$
      and hence since $s \geq 7,$ we find again 
     $$ M_{T, s-1}(I^2) \leq { C R^2}.$$ 
      By combining the last estimates and \eqref{est1:1}, we thus obtain \eqref{youyou}.
       This ends the proof of Proposition \ref{propzeta}.  
\end{proof}
It remains to prove Lemma \ref{lemvolterra}.
\begin{proof}[Proof of Lemma \ref{lemvolterra}]

 Let us  take $T>0$, and let us set for the purpose of the proof $K(t)= K(n,t)$, 
  $ F(t) = F_{n}(t) \mathds{1}_{ 0 \leq t  \leq T}. $ Since we  only consider the cases $n= \pm 1$, we do not write down anymore explicitly the dependence in $n$.  We consider the equation
 \begin{equation}
 \label{eqvolterra2}
   y(t)= K *  y (t)  +  F(t), \quad t \in \mathbb{R}
 \end{equation}
setting $y(t) = 0$ for $t \leq 0$. 
 Note that the solution of this equation coincides with $ \zeta_n(t)$  on $[0, T]$ since the modification of the source term for $t \geq T$ does not affect the past. 
     By taking the Fourier transform in $t$ (that we still denote by $\hat \cdot$\, ), we obtain
 \begin{equation}
 \label{voltfourier}  \hat y (\tau)=  \hat{K}(\tau) \hat y (\tau) + \hat F(\tau), \quad \tau \in \R, 
 \end{equation}
with  $ \hat{K}(\tau)= \hat{K}(n, \tau)$. Under the assumption ${\bf (H)}$, the solution of \eqref{voltfourier} is given explicitely by the formula
   \begin{equation}
   \label{vol2} \hat{y}(\tau)= { \hat{F}(\tau) \over 1 - \hat{K}(\tau)}.
   \end{equation} 
  Let us observe that  since $( 1 + v^2) \eta_{0} \in \Hc^5$,   we have by \eqref{emb1} that for $\alpha \leq 2$ and  for $t > 0$
    \begin{equation}
    \label{decvol2}| \partial_{t}^\alpha  K(t)   | \leq {C \over \langle t \rangle^4}    \in L^1(\mathbb{R}_{+}).
    \end{equation}
    Note that by definition of $K(t)$, the function $K(t)$ is continuous in $t = 0$, but not $C^1$. 
     Using an integration by parts on the definition of the Fourier transform, we  then get 
      that
     \begin{equation}
    \label{vol1}
      |  \partial_{\tau}^\alpha \hat{K}(\tau) | \leq { C \over \langle \tau \rangle^2}, \quad  \alpha  \leq 2.
      \end{equation}
 To get this, we  have used that the function $t\, \hat \eta_{0}(t)$ vanishes at zero. 
 
  By using this estimate on $\hat{K}$,  {\bf (H)}  and that  $\hat F(\tau) \in H_\tau^{1}$ (the Sobolev space in $\tau$) since $F$ is compactly supported in time, we easily
   get that $ y$ defined via its Fourier transform  by  \eqref{vol2}  belongs to $H_{\tau}^1$. This  implies that $\langle t \rangle y \in L^2$ and thus that $y \in L^1_{t}$.
  These remarks justify the  use  of the Fourier transform and that the function $y$ defined through its Fourier transform via \eqref{vol2}
   is a solution of  \eqref{eqvolterra2}. Moreover, thanks to \eqref{vol2} and {\bf (H)}, we  get that $\hat y$ can be continued 
    as an holomorphic function  in $\mathrm{Im}\, \tau \leq 0$ and  thanks to a Paley Wiener type  argument, that $y$ vanishes for $t \leq 0$.  We have thus obtained  an $L^1$ solution
     of \eqref{eqvolterra2} that vanishes  for $t \leq 0$. By a Gronwall type argument, we easily get  that there is a unique solution in this class  of \eqref{eqvolterra2} and thus
     we have obtained the expression of the unique solution.    
 
  We can thus now focus on the proof of the estimate stated in Lemma \ref{lemvolterra}.
      Note that  a $L^2$-based version of this estimate would be very easily obtained.   
The difficulty here is  to get the uniform   $L^\infty$ in  time  estimate we want to prove. 

\medskip 

   We shall first prove the estimate for $\gamma =0$.
Let us  take $\chi(\tau) \in [0, 1]$ a smooth compactly supported function  that  vanishes for $| \tau | \geq 1$ and which is equal to one for $|\tau| \leq 1/2$.   We  define $\chi_{R}(\tau)= \chi(\tau/R)$ and $\chi_R(\partial_t)$ the corresponding operator in $t$ variable corresponding to the convolution with the inverse Fourier transform of $\chi_R(\tau)$. 
  Thanks to \eqref{vol1}, we have that for $R$ large
 $$ \langle t \rangle^2 | (1- \chi_{R}(\partial_{t})) K (t) |  \leq C  \sum_{\alpha \leq 2}\|  \partial_{\tau}^\alpha ( ( 1 - \chi_{R}(\tau)) \hat{K}(\tau)) \|_{L^1(\R)}
  \leq C \int_{|\tau| \geq R/2} { 1 \over \langle \tau \rangle^2} \leq  { C \over R} $$
  and hence
  \begin{equation}
  \label{vol2'}
   \|  (1- \chi_{R}(\partial_{t} ) ) K(t) \|_{L^1(\R)} \leq {C \over R} \leq { 1 \over 2}
  \end{equation}
   for $R$ sufficiently large. This choice fixes $R$.
   
   To estimate the solution $y$ of \eqref{eqvolterra2}, we shall write that
    $$y=  \chi_{2R} (\partial_{t}) y + ( 1 - \chi_{2R}(\partial_{t}))y   = : y^l + y^h.$$
    By applying $ ( 1- \chi_{2R}(\partial_t)) $ to \eqref{eqvolterra2}, we get that
    $$
     y^h = K * y^h + ( 1 - \chi_{2R}(\partial_{t}) )  F =\big( ( 1 - \chi_{R}(\partial_{t}) K \big) * y^h +   ( 1 - \chi_{2R}(\partial_{t})  )F
    $$
     since $(1- \chi_{R})= 1$ on  the support of $1- \chi_{2R}$. Therefore, we obtain thanks to \eqref{vol2'} and the fact that $\chi_{2R}(\partial_t)$ is a convolution operator with a $L^1$ function,  that
     $$ \|y^h \|_{L^\infty} \leq  { 1 \over 2} \| y^h \|_{L^\infty} +  C \| F \|_{L^\infty}$$
      and hence
      $$ \|y^h \|_{L^\infty} \leq   2 C \|F\|_{L^\infty}.$$
       For the low frequencies, we can use directly the form \eqref{eqvolterra2} of the equation: We can write that
       $$  \hat y^l(\tau)= { \chi_{2R}(\tau)   \over 1 - \hat{K}(\tau)} \chi_{R}(\tau) \hat F(\tau).$$
        Since the denominator does not vanish thanks to  ${\bf(H) }$, we obtain again  that $y^l$ can be written as the convolution of an $L^1$ function - which is the inverse Fourier transform of $ \chi_{2R}(\tau)/( 1-  \hat{K}(\tau))$ -  by the function
         $\chi_{R}(\partial_{t}) F$ which is a convolution of $F$ by a smooth function. Thus we obtain by using again the Young inequality  that
         $$  \| y^l \|_{L^\infty} \leq  C \|F\|_{L^\infty}.$$
          Since 
          $ \|y \|_{L^\infty} \leq \| y^l \|_{L^\infty} + \|y^h\|_{L^\infty},$
           we get the  desired estimate  for $\gamma=0$.    
    To get the estimate for arbitrary $\gamma$, we can proceed by induction.  We observe that
    $$ t y(t) =  K * (t y) +  F^1$$
     with $F^1= (tK) * y + t  F$. 
     Using the result $\gamma = 0$, we obtain that $\| t  y \|_{L^\infty} \leq C \|F^1 \|_{L^\infty}$. Now 
      since $\eta_{0} \in \Hc^{\gamma + 3}$,  for $\gamma =1$, we obtain that $ tK  \in L^1$ and thus 
      $$  \|F^1 \|_{L^\infty} \leq C\big( \| t F\|_{L^\infty} +  \|y\|_{L^\infty}) \leq C \| ( 1 + t ) F\|_{L^\infty}.$$
       The higher order estimates follow easily in the same way.     
\end{proof}

\subsection{Estimate of $N_{T, s}(g)$}
\begin{proposition}
 \label{propgs}
 Assuming that $\eta \in \Hc^{s+2}$ verifies the assumption {\bf (H)}, then there exists $C>0$ such that for every $T>0$, every  solution of \eqref{eq:vp3}  such that
 $ Q_{T, s} (g) \leq R$ enjoys the estimate
   $$  N_{T, s}(g) \leq  C (R_{0}+  \eps R^2)  ( 1 + \eps R) e^{ C \eps  R}.$$
 \end{proposition}
 
 \begin{proof}
  To prove Proposition \ref{propgs}, we shall use energy estimates.
  We set $\mathcal{L}_t[g]$ the operator
 $$\Lc_t[g] f = \{Ê\phi(t,g),f\}$$ such 
 that $g$ solves the equation 
 $$\partial_t g = \Lc_t[g]( \eta + \varepsilon g).$$
For any linear operator $D$, we thus have by standard manipulations 
 that 
\begin{eqnarray*}
\frac{\dd}{\dd t} \Norm{Dg(t)}{L^2}^2 &=& 2 \varepsilon  \langle D g(t), D( \Lc_t [g] g(t)) \rangle_{L^2} + 2 \langle Dg(t), D( \Lc_t [g](\eta) )\rangle_{L^2}\\
&=& 2\varepsilon \langle Dg(t),  \Lc_t [g] D g(t) \rangle_{L^2}  + 2\varepsilon \langle Dg(t), [D, \Lc_t[g]] g(t)\rangle_{L^2}\\
&& + 2 \langle Dg(t), D( \Lc_t [g](\eta) )\rangle_{L^2},
\end{eqnarray*}
where $[D,\Lc_t]$ denotes the commutator between the two operators $D$ and $\Lc_t$.  
The first term in the previous equality vanishes since $\mathcal{L}_{t}[g]$ is the transport operator associated with a divergence free Hamiltonian vector field.  
Consequently,  we get that  
\begin{equation}
\label{energie1}
 {\dd \over \dd t}\Norm{Dg(t)}{L^2}^2 \leqslant    2\varepsilon\ \Norm{Dg(t)}{L^2} \Norm{[D,\Lc_t[g]] g(t)}{L^2} 
+ 2 \int_0^t \Norm{ Dg(t)}{L^2}\Norm{ D ( \Lc_t [g] (\eta) ) }{L^2}. 
\end{equation}
 To get the estimates of Proposition \ref{propgs}, we shall use the previous estimates with the operator 
$D = D^{m,p,q}$ defined as the   Fourier  multiplier by  $ k^p \xi^q  \partial^m_\xi$ for  $(m,p,q) \in \mathbb{N}^{3d}$ such that $p + q \leqslant s $, $m \leqslant m_0$ and the definition \eqref{defsob} of the $\Hc^s$ norm.  
  To evaluate the right hand-side of \eqref{energie1}, we shall use
  \begin{lemma}
  \label{lemcom}
   For   $p + q \leqslant  \gamma$ and $m \leqslant m_0$, and functions $h(t)$ and $g(t)$, we have the estimates
   \begin{eqnarray}
   \label{com1}
    & &  \| \big[ D^{m,p,q} ,  \mathcal{L}_{\sigma}[g] \big] h  (\sigma)\|_{L^2} \leq C \big( m_{\sigma, \gamma + 1}(\zeta)  \|h (\sigma)\|_{\Hc^1} +  m_{\sigma, 2}(\zeta) \|h(\sigma)\|_{\Hc^\gamma}\big),  \\
  \label{com2}   & & \|D^{m,p,q} \big(   \mathcal{L}_{\sigma}[g] \big)h  (\sigma) \|_{L^2} \leq C \big(  m_{\sigma, \gamma + 1}(\zeta)  \|h (\sigma)\|_{\Hc^{1}}
  + m_{\sigma, 2}(\zeta) \|h(\sigma)\|_{\Hc^{\gamma + 1 }},
      \end{eqnarray}
      for all $\sigma$, 
   where  $\zeta$ is still defined by $\zeta_{k}(t)= \hat g_{k}(t, kt)$, $k \in \{Ê\pm 1\}$, and where
   $$  m_{\sigma, \gamma}(\zeta) = \langle \sigma \rangle^{\gamma} \Big(\sup_{k \in \{\pm 1\}}  |\zeta_k(\sigma)|\Big),$$
   with a constant $C$ depending only on $\gamma$, and in particular, does not depend on  $\sigma$.
   \end{lemma} 
  Let us finish first the proof of  Proposition \ref{propgs}.  By using the previous lemma with $\gamma = s$ and \eqref{com1} with $h= g$ and \eqref{com2}
   with $h= \eta$, we obtain from   \eqref{energie1}    that
  \begin{multline*}  { \dd \over \dd t } \|g(t) \|_{\Hc^s}^2  \leq  \langle t \rangle^2 m_{t, s-1}(\zeta) \big( \| \eta \|_{\Hc^1} + \eps  \|g(t)\|_{\Hc^1} \big)  \|g(t)\|_{\Hc^s}  \\+   { 1 \over \langle t \rangle^{s - 3}} m_{t, s-1}(\zeta)
   \| \eta\|_{\Hc^{s+1}} \| g(t) \|_{\Hc^s} 
   +  { \eps  \over \langle t \rangle^{s - 3}} m_{t, s-1}(\zeta) \|g(t) \|_{\Hc^s}^2.
   \end{multline*}
    This yields using the fact that $M_{t,\gamma}(\zeta) = \sup_{\sigma \in [0,t]} m_{\sigma,\gamma}(\zeta)$,
  $$ \|g(t)\|_{\Hc^s} \leq \|g(0)\|_{\Hc^s} + \langle t \rangle^3 M_{t, s-1}(\zeta) \big(  \|\eta\|_{\Hc^{s+1}}+ \eps R \big) + \eps R \int_{0}^t { 1 \over \langle \sigma \rangle^{s - 3}}
   \| g(\sigma) \|_{s} \, \dd\sigma$$
    for $t \in [0, T]$.
     From the Gronwall inequality, we thus obtain
     $$ \|g(t)\|_{\Hc^s} \leq  \Big(\|g(0)\|_{\Hc^s} + \langle t \rangle^3 M_{t, s-1}(\zeta) \big(  \|\eta\|_{\Hc^{s+1}}+ \eps R \big)\Big) e^{\eps R \int_{0}^{ + \infty}{ \dd \sigma \over \langle \sigma \rangle^{s-3}}}.$$
      By using Proposition \ref{propzeta}, this yields
      $$  N_{T, s}(g) \leq \Big( R_{0} +  (R_{0}+  \eps R^2)  ( C + \eps R)  \Big) e^{ C\eps  R}.$$
       This ends the proof of Proposition \ref{propgs}.  

 \end{proof}

Let us give the proof of Lemma \ref{lemcom}.
\begin{proof}[Proof of Lemma \ref{lemcom}]
 We give the proof of \eqref{com1}, the proof of the second estimate  being  slightly easier. In the Fourier side, we have   for $\mathcal{L}_{\sigma}[g] (h)$ the expression
$$
 (  \mathcal{F}{\Lc_\sigma[g] h})_n(\xi) = \sum_{k \in \{ \pm 1\} }  kp_k  \zeta_k(\sigma)  \hat h_{n-k}(\sigma,\xi - k\sigma) (  n  \sigma -  \xi).
$$
  Consequently, we obtain that
  \begin{multline*} \big(  \mathcal{F}  ( [D^{m,p,q},  {\Lc_\sigma[g]  }h ) \big))_n(\xi) = \\
    \sum_{k \in \{ \pm 1\} }
      kp_k  \zeta_k(\sigma) \Big( n^p \xi^q  \partial^m_\xi \big( \hat h_{n-k}(\sigma,\xi - k\sigma) (  n \sigma-  \xi) \big) -  \\
        \big( (n-k)^p  (\xi-  \sigma)^q   \partial_{\xi}^m\hat h_{n-k} ( \sigma,\xi - k\sigma) (  n \sigma -  \xi) \big)  \Big).
       \end{multline*}
For $k =\pm 1$, we can thus expand the above expression into a finite sum of terms under the form
$$
 I_{n}^k(\sigma, \xi)=  k p_{k} \zeta_{k} (\sigma)  k^{p_{1}}  (n-k)^{ p - p_{1} + \alpha }  \big(k \sigma \big)^{q_{1}+ \alpha} (\xi - k\sigma)^{q - q_{1} + \beta}
 \partial_{\xi}^{m_{1}} \hat h_{n-k}(\sigma,\xi - k\sigma)
$$
where 
$$0 \leq   p_{1} \leq p, \,    0  \leq  q_{1} \leq q ,  \quad  m-1 \leq m_{1} \leq m, \quad   \alpha + \beta =  m_{1} - m  + 1, \, \alpha, \, \beta \geq 0.$$
Moreover, if $m_{1}=m$, then we have $p_{1}+ q_{1}>0$.

 We have to estimate
 $ \sum_{n} \int_{\xi}  |\sum_{k \in \pm 1}I_{n}^k (\sigma, \xi) |^2 \, d\xi$ by isometry of the Fourier transform.

 We note that for a fixed $k \in \{\pm 1\}$ then for $|n - k| + | \xi - k \sigma| \leq |k| \sigma$, we have 
 $$ | I_{n}^k (\sigma, \xi) | \leq  C \sigma^{ p + q + 1 } |\zeta_{k}(\sigma)| |n-k|| \partial_{\xi}^{m_1} \hat h_{n-k} (\sigma, \xi - k \sigma) |$$
  whereas for $|n-k| + | \xi- k \sigma | \geq  |k| \sigma$, we have
  $$| I_{n}^k (\sigma, \xi) |  \leq  C  \langle \sigma \rangle^2 |\zeta_{k}(\sigma)| (|n-k| + | \xi- k \sigma |)^\gamma | \partial_{\xi}^{m_1} \hat h_{n-k} (\sigma, \xi - k \sigma) |. $$ 
 
  Consequently  by taking the $L^2$ norm, we find that
 $$ \| \sum_{k \in \pm 1}I_{n}^k(\sigma,\xi) \|_{L^2} \leq  C \big( m_{\sigma, \gamma + 1}(\zeta)  \|h(\sigma) \|_{\Hc^1} +  m_{\sigma, 2}(\zeta) \|h(\sigma)\|_{\Hc^m}\big).$$
 This ends the proof of the  Lemma.
    \end{proof}
   
   \subsection{Estimate of $\| g \|_{ \Hc^{s-4}}$}
   To close the argument, it only remains to estimate  $\| g \|_{ \Hc^{s-4}}$.

    \begin{proposition}
    \label{propgs-4}
    Assuming that $\eta \in \Hc^{s+2}$ verifies the assumption {\bf (H)}, then there exists $C>0$ such that for every $T>0$, every  solution of \eqref{eq:vp3}  such that
 $ Q_{T, s} (g) \leq R$ enjoys the estimate
    $$ \| g(t) \|_{\Hc^{s-4}}  \leq  C\big( R_{0} +  \eps R^2) e^{ C \eps R}, \quad \forall t \in [0, T].$$
    \end{proposition}

 \begin{proof}
 We use again \eqref{energie1} with $D= D^{m,p,q}$ but now with $p+q  \leq s-4$.
  By using Lemma \ref{lemcom}, we find
  \begin{equation}
  \label{gs-4scat}
   {\dd \over \dd t} \|g(t)\|_{\Hc^{s-4}}^2 \leq  m_{t, s-3}(\zeta) \big(  \| \eta\|_{\Hc^{s-3}} \|g(t)\|_{\Hc^{s-4}} +\eps \|g(t) \|_{\Hc^{s- 4}}^2  \big).\end{equation}
   This yields
   $$ \|g(t)\|_{\Hc^{s-4}} \leq \|g(0)\|_{\Hc^{s-4}}  +  \| \eta \|_{\Hc^{s-3}} M_{t, s-1}(\zeta) \int_{0}^t { 1 \over \langle \sigma \rangle^2} \, \dd\sigma + \eps M_{t, s-1}(\zeta)
    \int_{0}^t { 1 \over \langle \sigma \rangle^2}  \|g(\sigma) \|_{\Hc^{s-4}} \, \dd \sigma.$$
    By using Proposition \ref{propzeta}, we thus get
    $$ \|g(t)\|_{\Hc^{s-4}}    \leq C\big( R_{0} +  \eps R^2) + \eps R  \int_{0}^t  { 1 \over \langle \sigma \rangle^2}\|g(\sigma) \|_{\Hc^{s-4}} \, \dd \sigma.$$
    From the Gronwall inequality, we  finally find
    $$ \| g(t) \|_{\Hc^{s-4}}  \leq C\big( R_{0} +  \eps R^2) e^{ C \eps R}.$$
     This ends the proof of Proposition \ref{propgs-4}.
 
 \end{proof}
 
 \section{Proof of Theorem \ref{maintheo}}
 The proof of Theorem \ref{maintheo} follows from the a priori estimates in Propositions \ref{propzeta}, \ref{propgs} and \ref{propgs-4} and a continuation argument.
  Indeed, by combining the estimates of these three propositions, we get that
  $$ Q_{T, s}(g) \leq  C (R_{0}+  \eps R^2)  ( 1  + \eps R)  e^{ C \eps  R}$$ 
  assuming that $Q_{T, s}(g) \leq R$. Consequently, let us choose $R$ such that
    $R> C R_{0}$, then for $\eps$ sufficiently small we have $ R >C (R_{0}+  \eps R^2)  ( 1  + \eps R)  \Big) e^{ C \eps  R}$ and hence by usual continuation argument, we obtain that
    the estimate $Q_{T, s}(g) \leq R$ is valid for all times.
    
   \section{Proof of Corollary \ref{Landau}}
   \label{ouf}
     In view of \eqref{fourier1}, let us define $g^\infty(x,v)$ by
     $$
     g^\infty(x,v) = g(0,x,v) + \int_0^{+\infty} \{Ê\phi(\sigma,g),  \eta + \varepsilon g(\sigma)\} \, \dd \sigma. 
     $$
     Note that the integral is convergent in $\Hc^{s-4}$ since thanks to  \eqref{com2}, we have
     $$ \|\{Ê\phi(\sigma,g),  \eta + \varepsilon g(\sigma)\} \|_{\Hc^{s-4}} \leq  C(R)\big( {1 \over \langle \sigma \rangle^{2}} + { \langle{ \sigma} \rangle^{3\over 4} \over  \langle \sigma \rangle^{s -3}}\big).$$
     Note that for the last estimate, we have used that by interpolation
      $$ \|g\|_{\Hc^{s-3}} \leq C \|g \|_{\Hc^{s-4}}^{3 \over 4} \| g \|_{\Hc^s}^{1 \over 4} \leq C(R)   \langle \sigma \rangle^{ 3 \over 4}.$$
       From the same arguments, we also  find that
      $$ \| g(t) - g^{\infty}\|_{\Hc^{s-4}} \leq C(R) \Big( \int_{t}^{+ \infty} { 1 \over \langle \sigma \rangle^{  2  } } +  { 1 \over   \langle \sigma \rangle^{s -3 - { 3 \over 4}} } \, \dd\sigma \Big)
       \leq {C(R)  \over  \langle  t  \rangle }.
       $$ 
      In a similar way, by using again \eqref{com2}, we have for $r \leq s-4$ and $r \geq 1$, 
   $$   \| g(t) - g{\infty}\|_{\Hc^{r}} \leq C(R) \Big( \int_{t}^{+ \infty} { 1 \over \langle \sigma \rangle^{  s- r - 2  } } +  { 1 \over   \langle \sigma \rangle^{s -3}  }\, \dd\sigma \Big)
       \leq C(R) \big( { 1 \over  \langle  t  \rangle^{ s- r - 3 }  } +  { 1 \over   \langle  t  \rangle^{s - 4} } \big) \leq  {C(R)  \over  \langle  t  \rangle^{ s- r - 3 } } .$$ 
       
%
%
%
   
 \section{The case of a kernel with a finite number of modes}
 \label{sectionfinite}
  In this section, we briefly indicate the modifications in the case that in \eqref{eq:hmf1}, the kernel $P$ is defined by
  $$ P(x)= \sum_{k=1}^M p_{k} \cos(kx), $$
  for some $p_k \in \R$ and for a fixed $M$. 
   For the Penrose criterion {\bf (H)} , it suffices to consider that it holds for any $n$,  $|n| \leq M, $ $n \neq 0$.
    
    We can use the weighted norms 
  $$  Q_{T,s}(g)=  \sup_{t \in  [0, T]} { \|g(t) \|_{\Hc^s} \over \langle t \rangle^{2M+1}} + 
   \sup_{t \in [0,T]} \sup_{ |k|  \leq M, \, k \neq 0} \langle t \rangle^{s+ 1 - 2k}|\zeta_{k}(t) |  +  \sup_{t \in [0, T ]} \|g(t) \|_{\Hc^{s-2 M - 2}}.$$
    One can then obtain that 
    \begin{theorem}
Let us fix $s  \geq  \max ( 4 M + 2, 2M+5) $  and $R_{0}>0$ such that $Q_{0, s}(g) \leq R_{0}$ and assume that $\eta\in \Hc^{s+4}$ satisfies the assumption ${\bf (H)}$.  Then there exists $R>0$ and $\eps_{0}>0$ such that for every $\eps \in (0, \eps_{0}]$ and
 for every $T\geq 0$, we have the estimate
 $$ Q_{T, s}(g) \leq R.$$
\end{theorem}
 It is then easy to get from this result a nonlinear damping effect as previously.
 
  The proof of   this result follows exactly  the same lines as the proof of Theorem \ref{maintheo}. The  estimates for
   $ \sup_{t \in  [0, T]} { \|g(t) \|_{\Hc^s} \over \langle t \rangle^{2M+1}}$ and  for $  \sup_{t \in [0, T ]} \|g(t) \|_{\Hc^{s-2 M - 2}}$ can be obtained exactly in the same way
    as in Proposition \ref{propgs}  and Proposition \ref{propgs-4}. The only technical difference is that in Lemma \ref{lemcom}, we  define
    $$m_{\sigma, \gamma }(\zeta)= \langle \sigma \rangle^\gamma \sup_{ |k| \leq M, \, k \neq 0 } | \zeta_{k}(\sigma)|.$$
    
     The only part were we need to be careful is to estimate $\sup_{t \in [0,T]} \sup_{ |k|  \leq M, \, k \neq 0} \langle t \rangle^{s+ 1 - 2k}|\zeta_{k}(t) |$
      as in Proposition \ref{propzeta} since more resonances are possible in the integral equation \eqref{zetanp1}.  By  using the Volterra equation \eqref{volterra1} for $|n| \leq M$, $n \neq 0$, we still get that
      $$ \sup_{t \in [0, T]} \langle t \rangle^{s+ 1 - 2n }  | \zeta _{n}(t) |\leq   \sup_{t \in [0, T]} \langle t \rangle^{s+ 1 - 2n }  |  F_{n}(t)|$$
       and we only need to estimate the right hand side. 
       
     The only difficulty is to estimate the contribution of the integral terms
      $$ J_{n}=  \sup_{t \in [0, T]} \langle t \rangle^{s+ 1 - 2n } 
        \sum_{ | k| \leq M, \, k \neq 0}  |p_k|  \int_0^t  |\zeta_k(\sigma) | \, |  \hat g_{n-k}(\sigma, nt - k\sigma)  | | k n|(t-\sigma)  Ê\dd \sigma$$
         for $ | n |  \leq M$.
        
       If $k$ and $n$ have opposite sign, then, we can proceed as in the estimate of $F_{n}^1$ in the proof of Proposition \ref{propzeta}, we find
    \begin{multline*}    \langle t \rangle^{s+ 1 - 2n }  \int_0^t  |\zeta_k(\sigma) | \, |  \hat g_{n-k}(\sigma, nt - k\sigma) | |k n|(t-\sigma) Ê\dd \sigma \\ \leq  C  Q_{t,s}(g)^2
     \langle t \rangle^{s+ 1 - 2n }  \int_{0}^t { \langle \sigma \rangle^{ 1+ 2M} (t- \sigma) \over \langle  \sigma \rangle^{ s+ 1 - 2k} \langle t \rangle^{ s} }\, \dd \sigma
      \\  \leq  C Q_{t, s}(g)^2 \langle t\rangle^{2 - 2 n } \int_{0}^{+ \infty} { 1 \over \langle \sigma \rangle^{ s- 4M} } \, \dd \sigma
      \end{multline*}
      which is uniformly bounded since $ s \geq 4M + 2$ and $ |n| \geq 1$.
      
      Now let us assume that $k$ and $n$ have the same sign. We can assume that  $k \geq 1$ and  $n \geq 1$, the other situation being similar. 
         If $n>k$, then we have that $nt- k \sigma \geq (k+1) t - k \sigma \geq  t$ and hence we can use the same bound as above.
         If $n=k$, we can proceed exactly as for the term  $F_{n}^2$ in the proof of Proposition \ref{propzeta}.
       It remains to handle the case $n<k$ which is new.
      For this one,  we split the time integral in the region $ \sigma \leq  {n \over 2 k} t$ and the region
       $ \sigma \geq   {n \over 2 k} t$.
        For the first region we have $ nt - k \sigma \geq { nt /2}$ and hence this part of the integral can be handled as previously.
       For the region   $\sigma \geq   {n \over 2 k} t$, we estimate it by  
      \begin{multline*}  \langle t \rangle^{s+ 1 - 2n }  \int_0^t  |\zeta_k(\sigma) | \, |  \hat g_{n-k}(\sigma, nt - k\sigma) | |k n|(t-\sigma) Ê\dd \sigma \\  \leq  C  Q_{t,s}(g)^2
       \langle t \rangle^{s+ 1 - 2n } \int_{{n \over 2 k} t}^t { t  \over \langle \sigma \rangle^{ s+ 1 - 2k}}\, \dd\sigma   \leq  C  \langle t \rangle^{2n - 2 k + 2 }.
       \end{multline*}
        This term is uniformly bounded since  $1 \leq n \leq k-1$.  
   
%
%
%
%
%

\end{document}